\documentclass[12pt]{amsart}
\usepackage{amsmath,amssymb,latexsym,amsmath,amsthm,amscd}
\usepackage{setspace}

\usepackage[left=3cm,top=2cm,right=3cm,bottom = 2cm]{geometry}
\usepackage{graphicx}
\usepackage[usenames,dvipsnames]{color}
\usepackage{hyperref}
\usepackage{enumerate}
\usepackage{charter}

\theoremstyle{plain}
\newtheorem{theorem}{Theorem}
\newtheorem{lemma}[theorem]{Lemma}
\newtheorem{proposition}[theorem]{Proposition}
\newtheorem{corollary}[theorem]{Corollary}

\newtheorem{definition}[theorem]{Definition}
\newtheorem{example}[theorem]{Example}

\newtheorem{remark}[theorem]{Remark}

\DeclareMathOperator{\vac}{\mathbf{1}}
\DeclareMathOperator{\End}{End}
\DeclareMathOperator{\Vir}{Vir}

\DeclareMathOperator{\Id}{Id}
\DeclareMathOperator{\Res}{Res} 
\DeclareMathOperator{\ad}{ad}
\DeclareMathOperator{\OPE}{\,\overset{OPE}{\sim}\,}

\newcommand{\R}{\mathbb{R}}
\newcommand{\C}{\mathbb{C}}

\newcommand{\Z}{\mathbb{Z}}

\newcommand{\nth}{\mathrm{th}}

\newcommand{\loc}[1]{\overset{#1}{\sim}}
\newcommand{\nprod}[3]{({#1}\,{
{*}_{{#2}}
}\,{#3})} 
\newcommand{\nord}[1] { \boldsymbol{\mathrm{:}}\,{#1} \, \boldsymbol{\mathrm{:}}}

\begin{document}
\title{Vertex Algebras According to Isaac Newton}
\author{Michael Tuite}
\address{School of Mathematics, Statistics and Applied Mathematics\\ 
National University of Ireland Galway, Galway, Ireland}
\email{michael.tuite@nuigalway.ie}
\keywords{Vertex Algebras, Conformal Field Theory, Finite Differences}

\begin{abstract} 
We give an introduction to vertex algebras using elementary forward difference methods originally due to Isaac Newton.
\end{abstract}
\maketitle
 
\section{Introduction}
In this paper we present an   introduction to the theory of vertex algebras \cite{B}, \cite{FLM}, \cite{ K}, \cite{LL}, \cite{FHL}, \cite{LZ}, \cite{ MN}, \cite{ MT}. A cursory examination of the literature of vertex algebras reveals a variety of identities involving binomial coefficients
\[
 \binom{n}{i}=\frac{n(n-1)\ldots (n-i+1)}{i!},
\]  
for all $n\in\Z$. We describe   how all of these arise from Newton's binomial theorem either directly or else  through elementary Newton finite difference identities \cite{N}.
In particular, our  approach provides 
both a motivation and a new understanding of the fundamental  axioms of locality and lower truncation for vertex operators.   
We also obtain a simplified and stronger proof of the Borcherds-Frenkel-Lepowsky-Meurmann identity.

\section{Newton Forward Differences and Formal Series}\label{Chapt_FDiff}
\subsection{Forward Differences}
We  consider an elementary but very relevant illustration of  formal series techniques used in vertex algebra theory. 
Our example comes from  Newton's theory of  finite  differences.\footnote{Borcherds \cite{B} defined vertex algebras whilst at Trinity College Cambridge  exactly 300 years after Newton \cite{N} invented finite  differences at the same institution!}  
\medskip

Let $U$ be a vector space  over a field of characteristic zero. Let  $ U^{\Z}$ denote the set of doubly infinite sequences $\alpha=\{\alpha_n\}_{n\in \Z}$  with components $\alpha_n\in U$. 
Define the \emph{(first) forward difference operator} $\Delta: U^{\Z}\rightarrow  U^{\Z}$
\begin{align}
(\Delta \alpha)_{n}=\alpha_{n+1}-\alpha_{n},\quad \alpha\in   U^{\Z}.
\label{eq:DeltaFDiff}
\end{align}
The \emph{$N^\mathrm{th}$ forward difference operator} is defined  for all integers $N\ge 2$ by 
\[
\Delta^N = \Delta\circ\Delta^{N-1}.
\]
\begin{example}
For a real function  $f(x)$ define $\alpha_n=f(n)\in \R$. Then $(\Delta \alpha)_{n}$  is the classical Newton forward difference used in the polynomial interpolation of $f(x)$.   
\end{example}

\noindent 
The action of $\Delta^N$ on $ U^{\Z}$ is given by:
\begin{lemma}
\label{lem:Delta}
The $N^\mathrm{th}$ forward difference of $\alpha\in  U^{\Z}$ has components
\begin{align}
(\Delta^N \alpha )_{n}=\sum_{i\ge 0} (-1)^i\binom{N}{i} \alpha_{n+N-i}.
\label{eq:Delta}
\end{align}
\end{lemma}
\begin{proof}
Write $\Delta = F-I$ where $F$ is the forward shift operator 
\begin{align}
(F \alpha)_n=\alpha_{n+1},
\label{eq:Fshift}
\end{align}
and  $I$ is the identity operator. 
The result follows from Newton's binomial identity
\[
(F-I)^N \alpha=\sum_{i\ge 0} \binom{N}{i} (-1)^i F^{N-i}\alpha.
\]
\end{proof}
\noindent We now consider $\ker\Delta^N$, the space of sequences with zero $N^\mathrm{th}$ forward difference. 
\begin{proposition}[Newton's Forward Difference Formula]
\label{prop:kerDel}
Let $\alpha\in U^{\Z}$ with components $\alpha_n$. 
If $\alpha\in\ker\Delta^N$ for some $N\ge 1$ then for all $n\in \Z$
\begin{align}
 \alpha_n =\sum_{i\ge 0} \binom{n}{i}\left(\Delta^i\alpha\right)_{0}.
\label{eq:alpha_poly}
\end{align}
Conversely, if $\alpha\in U^{\Z}$ has components
$\alpha_n =\sum\limits_{i=0}^{N-1} \displaystyle{\binom{n}{i}}R_i$ for $R_i\in U$ and some $N\ge 1$ then $\alpha\in\ker \Delta^N$.
\end{proposition}
\begin{proof}
Assume $\alpha\in\ker\Delta^N$. We have $\alpha_{n}=\left(F^{n}\alpha\right)_{0}$  for all $n\in\Z$ with $F$  the forward shift operator of  \eqref{eq:Fshift}. Then \eqref{eq:alpha_poly} follows from a binomial expansion of $F^n\alpha=(I+\Delta)^n\alpha$
\begin{align}
F^n\alpha=\sum_{i\ge 0}  \binom{n}{i}\Delta^i\alpha,
\label{eq:Fnalpha}
\end{align}
for all $n\in \Z$. For $n\ge 0$, \eqref{eq:Fnalpha} is obvious whereas for $n=-k<0$  we can verify that 
\[
\alpha=F^k \sum_{i\ge 0}  \binom{-k}{i}\Delta^i\alpha,
\]
for all $\alpha\in\ker\Delta^N$. Hence \eqref{eq:Fnalpha} holds and therefore \eqref{eq:alpha_poly} results. 

Conversely, if $\alpha_n =\sum_{i=0}^{N-1} {\binom{n}{i}}R_i$ then 
noting that for $\beta_n=n^k$ and  $k>0$
\begin{align*}
(\Delta \beta)_{n}=kn^{k-1}+O\left(n^{k-2}\right) ,
\end{align*}
we find $\alpha\in\ker\Delta^N$ since $ \binom{n}{i}=\frac{1}{i!}n^i+O\left(n^{i-1}\right)$. 
\end{proof}

\noindent  We also note the following result:
\begin{corollary}
\label{cor:poly}
$\alpha\in\ker\Delta^N$ iff $\alpha_n=p_{N-1}(n)$ where $p_{N-1}(n)$ is a degree $N-1$ polynomial in $n$  with coefficients in $U$. 
\end{corollary} 

\begin{example}
Let $p_{N-1}(x)$ be a polynomial of degree $N-1$ with coefficients in $\R$. 
Then Proposition~\ref{prop:kerDel} is  Newton's forward difference formula expressing $p_{N-1}(n)$ for all $n\in\Z$  in terms of $p_{N-1}(i)$ for $i=0,\ldots ,N-1$. 
Replacing $n$ by $x$ on the right hand side of \eqref{eq:alpha_poly} gives the  Newton  interpolating polynomial for a real function $f(x)$ in terms of $\alpha_i=f(i)$ for $i=0,1,\ldots ,N-1$.
\end{example}

\subsection{Formal Generating Series}
Define a    \emph{formal generating series} $\alpha(z)$ for 
$\alpha\in U^{\Z}$  by
\begin{align}\label{alphagen}
\alpha(z) = \sum_{n \in \Z} \alpha_nz^{-n-1} \in U[[z, z^{-1}]],
\end{align}
where $U[[z, z^{-1}]]$ denotes   the space of  formal Laurent series in an indeterminate parameter $z$ with coefficients in $U$.
We associate the component $\alpha_n$ with $z^{-n-1}$ for reasons that become  clearer below e.g. \eqref{eq:deltai}, Lemma~\ref{lem:kerDelGen} and Theorem~\ref{th:BFLM}.
\medskip

Define the \emph{formal derivative} $\partial$ of $\alpha(z)$  by
\begin{align}
\partial\alpha(z) = \sum_{n \in \Z} \alpha_{n}(-n-1)z^{-n-2} 
=\sum_{n \in \Z} \left(-n\alpha_{n-1}\right)z^{-n-1}.\label{delz}
\end{align}
We also define $\partial^{(i)}:=\dfrac{1}{i!}\,\partial^{i}$. 
The  \emph{formal residue} of the Laurent series \eqref{alphagen}  is defined  by
\begin{align}\label{Resz}
\Res_{z}\alpha(z) = \alpha_{0}.
\end{align}

\begin{lemma}\label{lem:LeibIP}
The formal series $\alpha(z)$ satisfies versions of the fundamental theorem of calculus, the Leibniz rule  and integration by parts:
\begin{align}
\label{eq:FTC}
&\Res_{z}\partial \alpha(z) =0,\\
\label{eq:Leibniz}
&\partial\left(z^k\alpha(z) \right)=kz^{k-1}\alpha(z)+z^k\partial\alpha(z),\\
&\Res_{z}z^k\partial\alpha(z) =-k\Res_{z}z^{k-1}\alpha(z),
\label{eq:ibyp}
\end{align}
for all $k\in\Z$.
\end{lemma}
 \begin{proof}
\eqref{delz} immediately   implies \eqref{eq:FTC} and \eqref{eq:Leibniz}. $\Res_{z}\partial \left(z^k\alpha(z) \right)=0$ implies \eqref{eq:ibyp}.  
\end{proof}

The formal nature of a  series $\alpha(z)$ is well-illustrated by the constant sequence
$\alpha_n=\alpha_0\in U$, for all $n$, for which
\[
\alpha(z) = \alpha_0\,\delta(z),
\]
for formal \emph{delta series} defined by
\begin{align}\label{eq:delta}
\delta(z) = \sum_{m \in \Z} z^{m}.
\end{align}
The delta series  is analogous to the Dirac delta function in the sense that 
\begin{align}
z^k\delta(z)=\delta(z),
\label{eq:fzdelta}
\end{align}
for all $k\in\Z$. In particular, we note that 
\begin{align}
(z-1)\delta(z)=0.
\label{eq:delta0}
\end{align}
We also define  a family of formal delta series indexed by  integers $i\ge 0$ as follows: 
\begin{align}\label{deltazi}
\delta^{(i)}(z) =&(-1)^i \partial^{(i)} \delta(z)
=\sum_{m \in \mathbb{Z}} (-1)^i\binom{m}{i}z^{m-i},
\end{align}
with $\delta^{(0)}(z) =\delta(z)$.  
On relabelling, we note that \eqref{deltazi} can be rewritten as 
\begin{align}\label{eq:deltai}
\delta^{(i)}(z) = \sum_{n \in \Z} \binom{n}{i}z^{-n-1},
\end{align}
i.e. $\delta^{(i)}(z) $ is the formal series for the integer sequence 
$\left\{\binom{n}{i}\right\}_{n\in \Z} $.
We further find that  \eqref{eq:delta0} generalises to:
\begin{lemma}
\label{lem:deltafun}
$(z-1)\delta^{(i)}(z)=\delta^{(i-1)}(z)$ for all $i\ge 1$.
\end{lemma}

Since $(F\alpha)(z)=z\alpha(z)$ for the forward shift operator of \eqref{eq:Fshift}, it follows that the formal series for $\Delta^i \alpha$ is 
\begin{align}
\label{eq:DeliGF}
\left(\Delta^i \alpha\right)(z)=(z-1)^i\alpha(z).
\end{align} 
Thus $\alpha\in\ker\Delta^N$ iff $(z-1)^N\alpha(z)=0$. 
Noting that 
\begin{align}
(\Delta^i \alpha)_{0}=\Res_z(z-1)^i\alpha(z) ,
\label{eq:Resalphai}
\end{align}
we may reformulate  Newton's forward difference formula Proposition~\ref{prop:kerDel} in terms of formal series using 
\eqref{deltazi} and 
\eqref{eq:Resalphai} to find:
\begin{lemma}
\label{lem:kerDelGen}
Let $\alpha\in U^{\Z}$. 
Then $\alpha\in\ker\Delta^N$ iff
\begin{align}
\alpha(z) =\sum_{i=0}^{N-1} R_i\delta^{(i)}(z),
\label{eq:alpha_delta}
\end{align}
for $R_i=\Res_z(z-1)^i\alpha(z) \in U$. 
\end{lemma}
\medskip

In numerous classical applications of generating series with $U= \C$, the formal parameter $z$ can be taken to be a complex number in some domain on which the generating series converges. 
However, the formal delta series $\delta^{(i)}(z)$  diverges everywhere on the complex plane.
Nevertheless,  if we decompose $\delta^{(i)}(z)=\delta^{(i)}(z)_{+}+\delta^{(i)}(z)_{-}$ with 
\begin{align}
\delta^{(i)}(z)_{+}= \sum_{n \ge 0} \binom{n}{i}z^{-n-1},\qquad 
\delta^{(i)}(z)_{-} = \sum_{n \le -1} \binom{n}{i}z^{-n-1},\label{deltapm}
\end{align}
where the $\pm$  subscripts refer to the sign of the sequence index $n$.
Then the  series $\delta^{(i)}(z)_{+}$ and $\delta^{(i)}(z)_{-}$ converge on disjoint complex domains as follows:
\begin{align*}
\delta^{(i)}(z)_{+} = \frac{1}{(z-1)^{i+1}},\quad |z|>1,\qquad
\delta^{(i)}(z)_{-} =\frac{-1}{(z-1)^{i+1}},\quad |z|<1.
\end{align*}
We  utilise these expansions for  formal $z$ by adopting the following convention:
\begin{definition}[Expansion Convention]\label{def:formexp}
For $m\in \Z$ and formal variables $x,y$  we define
\begin{align}
(x+y)^{m}=\sum_{k\ge 0}\binom{m}{k} x^{m-k}y^k, \label{eq:xypower}
\end{align}
i.e. we expand in the second variable.
For $m\ge 0$, $(x+y)^{m}=(y+x)^{m}$, with a finite sum, whereas for $m<0$, 
$(x+y)^{m}$ and $(y+x)^{m}$ are distinct infinite  series.
\end{definition}
Following this convention we  write 
\begin{align}\label{deltapm2}
\delta^{(i)}(z)_{+}= \frac{1}{(z-1)^{i+1}},\qquad
\delta^{(i)}(z)_{-} =-\frac{1}{(-1+z)^{i+1}},
\end{align} 
so that
\[
\delta^{(i)}(z)= \frac{1}{(z-1)^{i+1}}-\frac{1}{(-1+z)^{i+1}}.
\]
We may similarly decompose any formal series as $\alpha(z)=\alpha(z)_{+}+\alpha(z)_{-}$ where 
\begin{align}\label{eq:alphapm}
\alpha(z)_{+}= \sum_{n \ge 0} \alpha_{n} z^{-n-1},\qquad 
\alpha(z)_{-} = \sum_{n \le -1}  \alpha_{n} z^{-n-1}.
\end{align}
The $\pm$ subscripts   refer to the sign of the sequence index $n$.\footnote{We use the convention usually adopted in physics which is opposite to that chosen in  \cite{K}.}
Thus Lemma~\ref{lem:kerDelGen}  and \eqref{deltapm2} imply $\alpha\in\ker(\Delta^N )$ iff
\begin{align*}
\alpha(z)_{+} =\sum_{i=0}^{N-1}\frac{R_i}{(z-1)^{i+1}},\qquad 
\alpha(z)_{-} =-\sum_{i=0}^{N-1}\frac{R_i}{(-1+z)^{i+1}}.
\end{align*}
\medskip 
Altogether the following theorem summarises our discussion thus far.
\begin{theorem} 
\label{theor:alphaseq}
 Let $\alpha\in U^{\Z}$ with formal series $\alpha(z)$ and let
\[
R_i=\left(\Delta^i\alpha\right)_{0}=\Res_z(z-1)^i\alpha(z) \in U,
\]
for integers $i\ge 0$.  
The following are equivalent:
\begin{enumerate}[(i)]
\setlength\itemsep{2mm}
  \item $\alpha\in \ker \Delta^N$,
	\item $(z-1)^N\alpha(z)=0$,
	\item  $\alpha_n =\sum\limits_{i=0}^{N-1} \displaystyle{\binom{n}{i}}R_i$ for all $n\in\Z$,
	\item $\alpha_n=p_{N-1}(n)$, a degree $N-1$ polynomial in $n$ with coefficients in $U$,
	\item $\alpha(z) =\sum\limits_{i=0}^{N-1} R_i\delta^{(i)}(z)$,
	\item $\alpha(z)_{+} =\sum\limits_{i=0}^{N-1}\dfrac{R_i}{(z-1)^{i+1}}$ and 
	$\alpha(z)_{-} =-\sum\limits_{i=0}^{N-1}\dfrac{R_i}{(-1+z)^{i+1}}$.  
\end{enumerate}
\end{theorem}

\subsection{Calculus of Formal Series}
We gather a compendium of  results concerning formal series that we make use of later. 

\begin{lemma}[Taylor's Theorem] \label{lem:Taylor}
For formal series $\alpha(z)$ then $\alpha(x+y)$  has formal Taylor expansion in $y$
\begin{align}
\alpha(x+y) = e^{y \partial}\alpha(x),
\label{eq:Taylor}
\end{align}
where $e^{y \partial}:=\sum_{i\ge 0} y^i\,\partial^{(i)}$.
\end{lemma}
 \begin{proof} 
Using  the formal expansion convention Definition~\ref{def:formexp} we find 
\begin{align*}
\alpha(x+y) =& \sum_{n\in \Z}\alpha_{n} (x+y)^{-n-1}=
 \sum_{n\in \Z}\alpha_{n}\sum_{i\ge 0}\binom{-n-1}{i} x^{-n-1-i}y^i\\
=& \sum_{n\in \Z}\alpha_{n}\sum_{i\ge 0}{y^i}\partial^{(i)} \left(x^{-n-1}\right)
= \sum_{i\ge 0} {y^i}\,\partial^{(i)}\alpha(x). 
\end{align*} 
\end{proof} 
 \begin{lemma}
For  $\alpha(z)_{\pm }$ of \eqref{eq:alphapm} we have
\begin{align}
\partial\left(\alpha(z)_{\pm }\right) =\left (\partial\alpha(z)\right)_{\pm }.
\label{eq:delpm}
\end{align}
\end{lemma}

\begin{lemma} [Residue Theorem]\label{lem:Respm}
Let $\alpha(z)$ be a formal series. For integer $k\ge 0$ we have
\begin{align}
\Res_{x}\frac{\alpha(x)}{(x-z)^{k+1}}=\partial^{(k)}\alpha(z)_{-},\qquad 
\Res_{x}\frac{\alpha(x)}{(-z+x)^{k+1}}=-\partial^{(k)}\alpha(z)_{+}.
\label{eq:Reskpm}
\end{align}
\end{lemma}
 \begin{proof} 
Consider 
\begin{align*}
\Res_{x}\, \frac{\alpha(x)}{x-z}=& \Res_{x} \sum_{n\in \Z}\sum_{r\ge 0} \alpha_{n} x^{-n-r-2}z^r=\sum_{r\ge 0} \alpha_{-r-1}\, z^r=\alpha(z)_{-}.
\end{align*}
Similarly we find $\Res_{x}\dfrac{\alpha(x)}{-z+x}=-\alpha(z)_{+}$. Thus \eqref{eq:Reskpm} holds for $k=0$.
The general result follows on applying $\partial_z^{(k)}$ and using \eqref{eq:delpm}. 
\end{proof}

\bigskip

\section{Locality}  \label{Chapt_Locality}
  \subsection{Locality of Formal Series}
Let us now assume that $U$ is  an associative algebra over a field of characteristic zero i.e. $U$ is a vector space equipped with an associative bilinear product $AB\in U$ for all $A,B\in U$. 
In the next section we will consider $U$ to be the algebra of endomorphisms of a vector space $V$.

We define the formal product and commutator for formal generating series $\alpha(z),\beta(z)\in U[[z, z^{-1}]]$ by 
\begin{align}
 \alpha(x)\beta(y)&= \sum_{m,n\in \Z}\alpha_{m}\beta_{n}x^{-m-1}y^{-n-1}
\in U[[x, x^{-1}, y, y^{-1}]],
\label{eq:prodab}\\
\left[\alpha(x),\beta(y)\right] &= \sum_{m,n\in \Z}[\alpha_{m},\beta_{n}] x^{-m-1}y^{-n-1}
\in U[[x, x^{-1}, y, y^{-1}]],\;\;
\label{eq:comab}
\end{align}
for independent indeterminates $x$ and $y$ and commutator $[\alpha_{m},\beta_{n}]=\alpha_{m}\beta_{n}-\beta_{n}\alpha_{m}$. 
In the language of Chapter~\ref{Chapt_FDiff}, the bivariate series 
\eqref{eq:prodab} and \eqref{eq:comab} are  generating series for  \emph{doubly indexed }sequences $\{\alpha_{m}\beta_{n}\}_{m,n\in\Z}$ and $\{[\alpha_{m},\beta_{n}]\}_{m,n\in\Z}$,
respectively.
\medskip

We now define the fundamental notion of \emph{locality} -- one of the most important  properties enjoyed by vertex operators \cite{Li},\cite{G},\cite{LZ}. 
For formal series  $\alpha(x),\beta(y)$ and integers $n\ge 0$ we define the formal bivariate  series \footnote{This is a well-defined formal series in $x$ and $y$ since $n\ge 0$.}
\begin{align}
C^{n}(\alpha(x),\beta(y))=(x-y)^{n}[\alpha(x),\beta(y)].
\label{eq:Cxy0}
\end{align}
\begin{definition} [Locality]  \label{def:loc}
$\alpha(z), \beta(z) \in U[[z, z^{-1}]]$ are called \emph{mutually local}
if for some integer $n\ge 0$ 
\begin{align}\label{local1def}
C^{n}(\alpha(x),\beta(y))= 0.
\end{align}
\end{definition}
The \emph{order of locality}  of $\alpha(z)$ and $\beta(z)$ is the \emph{least} integer $n=N\ge 0$ for which \eqref{local1def} holds, in which case
we say that $\alpha(z)$ and $\beta(z)$ are \emph{mutually local of order $N$}
 and write $\alpha(z) \loc{N} \beta(z)$ (or simply $\alpha(z)\sim \beta(z)$
 if $N$ is not specified). 
We also say that $\alpha(z)$ is  \emph{local} if $\alpha(z) \sim \alpha(z)$. 
\begin{lemma}\label{lem:delalpha}
If $\alpha(z) \loc{N} \beta(z)$ then  $\partial\alpha(z) \loc{N+1}  \beta(z)$.
\end{lemma}
\begin{proof}
 $0=\partial_x C^{N+1}(\alpha(x),\beta(y))=NC^{N}(\alpha(x),\beta(y))+C^{N+1}(\partial\alpha(x),\beta(y))$.
\end{proof}
We define the \emph{$n^{\nth}$ residue product} ${*}_{n}$ 
  for $n\ge 0$ of formal series $\alpha(z),\beta(z)$ to be the formal series
 \footnote{The $n^{\nth}$ residue product is often also notated by $\alpha(z)_{(n)}\beta(z)$ e.g. \cite{K,MN}.}
\begin{align}
\nprod{\alpha}{n}{\beta}(z) &=\Res_x C^{n}(\alpha(x),\beta(z)) 
=\sum_{k=0}^{n} \binom{n}{k} (-z)^k\, [\alpha_{n-k},\beta(z)].
\label{eq:anb}
\end{align} 
For $\alpha(z) \loc{N} \beta(z)$ it follows that 
\begin{align}
\nprod{\alpha}{n}{\beta}(z)=0\, \mbox{ for all }\, n\ge N.
\label{eq:nthprodzero}
\end{align}

The locality condition \eqref{local1def} is closely related to Theorem~\ref{theor:alphaseq} of the last section  for an appropriate choice of vector space and sequence.
Let $W=U[[ y, y^{-1}]]$ be the vector space of formal series in $y$ with coefficients in $U$.
If $\alpha(z) \loc{N} \beta(z)$ then
\begin{align}
0&=y^{1-N}C^{N}(\alpha(x),\beta(y)) =  \left(z-1\right)^N \sum_{n\in \Z} \gamma_{n}z^{-n-1},
\label{eq:zNalphay}
\end{align}
for $z=\dfrac{x}{y}=xy^{-1}$ and 
\begin{align}
\gamma_{n}= y^{-n}\,[\alpha_{n},\beta(y)]. 
\label{eq:gamman}
\end{align}
$\gamma_n$  
determines a sequence $\gamma\in W^{\Z}$ with formal series 
\begin{align*}
\gamma(z)&=\sum_{n\in \Z} \gamma_{n} z^{-n-1}=y \,[\alpha(yz),\beta(y)],
\end{align*}
where $\alpha(yz)=\sum_{n\in \Z} \alpha_{n}(yz)^{-n-1}$. 
But \eqref{eq:zNalphay} implies
\[
(z-1)^N \gamma(z)=0,
\]
which is Property~(ii) of Theorem~\ref{theor:alphaseq}. Therefore $\gamma$  satisfies the  equivalent properties following from  Newton's forward difference formula. 
In particular, Theorem~\ref{theor:alphaseq}~(i) implies  $\gamma\in \ker \Delta^N$ which together with Lemma~\ref{lem:Delta} implies that for all $n\in \Z$
\begin{align}
\sum_{i=0}^{N} \binom{N}{i} (-y)^i\, [\alpha_{n-i},\beta(y)]=0.
\label{eq:Delalphay}
\end{align}
Theorem~\ref{theor:alphaseq}~(iii) determines $\gamma$ in terms of the $N$ residues 
\begin{align*}
R_i(y)&=\Res_z(z-1)^i\gamma(z) = y^{-i}\, \nprod{\alpha}{i}{\beta}(y),
\end{align*}
for $i^{\nth}$ residue product \eqref{eq:anb} with $0\le i\le N-1$.
Thus Theorem~\ref{theor:alphaseq}~(iii) implies locality is equivalent to
\begin{align}
[\alpha_m,\beta(y)] =\sum\limits_{i=0}^{N-1} \binom{m}{i}y^{m-i}
\nprod{\alpha}{i}{\beta}(y).
\label{eq:amby}
\end{align}
In terms of components,  \eqref{eq:amby} reads
\begin{align}
[\alpha_m,\beta_n] =\sum\limits_{i=0}^{N-1} \binom{m}{i}\nprod{\alpha}{i}{\beta}_{m+n-i}.
\label{eq:ambn}
\end{align}
Theorem~\ref{theor:alphaseq}~(iv)--(vi) describe further corresponding  properties equivalent  to locality. Recalling \eqref{eq:alphapm},   Theorem~\ref{theor:alphaseq}~(vi) implies
\[
[\alpha(x)_{+},\beta(y)] =\sum\limits_{i=0}^{N-1}\dfrac{\nprod{\alpha}{i}{\beta}(y)}{(x-y)^{i+1}},\qquad 
[\alpha(x)_{-},\beta(y)] =-\sum\limits_{i=0}^{N-1}\dfrac{\nprod{\alpha}{i}{\beta}(y)}{(-y+x)^{i+1}},
\]
employing the formal expansion convention \eqref{eq:xypower}.
Altogether we therefore find Theorem~\ref{theor:alphaseq} implies the following list of properties equivalent to locality \cite{K}:
\newpage
\begin{theorem}
\label{theor:ablocal}
Let $\alpha(z),\beta(z)\in U[[z, z^{-1}]]$  and  let  $\nprod{\alpha}{i}{\beta}(z)$ be the $i^{\nth}$ residue  product for $i\ge 0$.   
The following are equivalent:
\begin{enumerate}[(i)]
\setlength\itemsep{2mm}
\item  $\alpha(z) \loc{N} \beta(z)$,
  \item $\sum\limits_{i=0}^{N} \displaystyle{\binom{N}{i}}(-y)^i\, [\alpha_{n-i},\beta(y)]=0$ for all $n\in\Z$,
	\item  $[\alpha_m,\beta(y)] =\sum\limits_{i=0}^{N-1} \displaystyle{\binom{m}{i}}y^{m-i}\,\nprod{\alpha}{i}{\beta}(y) $,
	\item $[\alpha_m,\beta_n] =\sum\limits_{i=0}^{N-1} \displaystyle{\binom{m}{i}}\nprod{\alpha}{i}{\beta}_{m+n-i}$,
	\item $y^{-m}[\alpha_m,\beta(y)] =p_{N-1}(m)$, where $p_{N-1}(m)$ is a degree $N-1$ polynomial in $m$ with coefficients in $ U[[y, y^{-1}]]$,
	\item $[\alpha(x),\beta(y)]=\sum\limits_{i=0}^{N-1} y^{-i-1}\delta^{(i)}\left(\dfrac{x}{y}\right)\nprod{\alpha}{i}{\beta}(y)$,
	\item $[\alpha(x)_{+},\beta(y)] =\sum\limits_{i=0}^{N-1}\dfrac{\nprod{\alpha}{i}{\beta}(y)}{(x-y)^{i+1}}$ and 
	$[\alpha(x)_{-},\beta(y)] =-\sum\limits_{i=0}^{N-1}\dfrac{\nprod{\alpha}{i}{\beta}(y)}{(-y+x)^{i+1}}$. 
\end{enumerate}
\end{theorem}
\medskip

We also define the   \emph{normally ordered product} of $\alpha(x)$ and $\beta(y)$ by
\begin{align}
\nord{\alpha(x)\beta(y)} =\alpha(x)_{-} \beta(y)+\beta(y)\alpha(x)_{+}.
\label{eq:normab}
\end{align}
Thus  we find 
\begin{align*}
\alpha(x)\beta(y)&=[\alpha(x)_{+},\beta(y)]\;+\nord{\alpha(x)\beta(y)},\\
\beta(y)\alpha(x)&=-[\alpha(x)_{-},\beta(y)]\;+\nord{\alpha(x)\beta(y)},
\end{align*}
which imply 
\begin{corollary}(OPE)
\label{cor:abnormal}
Let $\alpha(z),\beta(z)\in U[[z, z^{-1}]]$.    
Then $\alpha(z) \loc{N} \beta(z)$ if and only if
\begin{align*}
\alpha(x)\beta(y)&=\sum\limits_{i=0}^{N-1}\dfrac{\nprod{\alpha}{i}{\beta}(y)}{(x-y)^{i+1}}+\nord{\alpha(x)\beta(y)},\\
\beta(y)\alpha(x)&=\sum\limits_{i=0}^{N-1}\dfrac{\nprod{\alpha}{i}{\beta}(y)}{(-y+x)^{i+1}}+\nord{\alpha(x)\beta(y)}.
\end{align*}
\end{corollary}
\begin{remark}
The above expressions for $\alpha(x)\beta(y)$ and $\beta(y)\alpha(x)$ are related to the \emph{Operator Product Expansion~(OPE)}  
in chiral conformal field theory (e.g. \cite{BPZ},\cite{FMS})
\begin{align}
\alpha(x)\beta(y)\OPE \sum\limits_{i=0}^{N-1}\frac{\nprod{\alpha}{i}{\beta}(y)}{(x-y)^{i+1}},
\label{eq:abOPE}
\end{align}
to indicate the ``pole structure'' in the ``complex domain'' $|x|>|y|$. 
The remaining ``non-singular parts''  are not displayed since the pole terms determine the commutation relations of the components  in Theorem~\ref{theor:ablocal}~(iii)--(vii).
We also note that $N$, the order of locality,  determines the highest pole order.
\end{remark}

\subsection{Examples of Locality}
Suppose that $U$ is a Lie algebra where 
for all $u,v,w\in U$ the commutator  satisfies the Jacobi identity 
\begin{align}
[[u,v],w]+[[v,w],u]+[[w,u],v]=0.
\label{eq:Jacobi}
\end{align}
\subsubsection{The Heisenberg  Algebra} 
Consider the vector space
\[
\widehat{H}=\bigoplus_{n\in \Z}\C h_n\oplus \C K,
\]
 with basis $h_n$ and central element $K$ obeying the Lie algebra commutation relations:
\begin{align}
\left[h_{m},h_n\right]=m\delta_{m,-n}K,\quad \left[h_{m},K\right]=0.
\label{eq:Heisalg}
\end{align}
The formal series $h(z)=\sum_{n\in\Z}h_{n}z^{-n-1}$ obeys
\[
y^{-m}\left[h_{m},h(y)\right]=y^{-m}\sum_{n\in\Z}\left[h_{m},h_n\right]y^{-n-1}=K\, \binom{m}{1}y^{-1},
\]
 a degree 1 polynomial in $m$. Thus Property~(v) of Theorem~\ref{theor:ablocal} holds
(with $\alpha=\beta=h$) which implies $h(z) $ is local of order $N=2$. Equivalently, we have Property~(iii) of 
Theorem~\ref{theor:ablocal} with
\[
\nprod{h}{0}{h}(y)=0,\quad 
\nprod{h}{1}{h}(y)=K. 
\]
Considering the Lie algebra components $ h_n$  as being elements of the universal enveloping algebra of $\widehat{H}$ we obtain the OPE
\begin{align}
h(x)h(y)\OPE \frac{K}{(x-y)^{2}}.
\label{eq:BosOP}
\end{align}

This example is known in vertex algebra theory as the Heisenberg or free boson algebra  and in conformal field theory as the bosonic string e.g. \cite{P},\cite{FMS}.

\subsubsection{Affine Kac-Moody Algebras}
Let $\mathfrak{g}$ be a finite dimensional Lie algebra with Lie bracket $[\ ,\ ]$ equipped with an invariant symmetric, bilinear form
$\langle \ , \ \rangle: \mathfrak{g} \otimes \mathfrak{g} \rightarrow \C$  i.e. 
\begin{align}
\langle [a,b], c \rangle = \langle a, [b,c] \rangle,
\label{eq:Lieinv}
\end{align}
for all $a,b,c\in \mathfrak{g} $.
The \emph{affine Lie algebra} or \emph{Kac-Moody algebra}
associated to $(\mathfrak{g}, \langle \ , \ \rangle)$ is the vector space
\begin{align*}
\widehat{\mathfrak{g}} = \mathfrak{g} \otimes \C\left[t, t^{-1}\right] \oplus \C K = \bigoplus_{n\in \Z} \mathfrak{g} \otimes t^n \oplus \C K,
\end{align*}
with Lie algebra commutators 
\begin{align}
\notag
&\left[a \otimes t^m, b \otimes t^n\right] = \left[a, b\right] \otimes t^{m+n} + m \langle a, b\rangle \delta_{m, -n}K, \\
\label{eq:KMalg}
 &\left[a\otimes t^m, K \right]=0,
\end{align}
for all $a, b \in \mathfrak{g}$. 
Define the formal series $a(z)=\sum_{n\in\Z}a_{n}z^{-n-1}$ for 
$a\in\mathfrak{g}$ where
$
a_n:= a \otimes t^n$. 
Then we find that for all $a, b \in \mathfrak{g}$  
\[
y^{-m}\left[a_{m},b(y)\right]=\left [a, b\right](y) +K\langle a, b\rangle \,\binom{m}{1}y^{-1}.
\]
This is Property~(iii) of  Theorem~\ref{theor:ablocal} for $N=2$ with
\[
\nprod{a}{0}{b}(y)=\left [a, b\right](y),\quad 
\nprod{a}{1}{b}(y)=K\langle a, b\rangle. 
\]
Hence $a(z)$ and $b(z)$ are mutually local of order 2 if $\langle a, b\rangle\neq 0$, of order 1  if $\langle a, b\rangle=0 $ and  $[a,b]\neq 0$ and of order $0$ if
$\langle a, b\rangle=0 $ and  $[a,b]=0$.
With a suitable  universal enveloping algebra interpretation we obtain the OPE
\begin{align}
a(x)b(y)\OPE \frac{K\langle a, b\rangle}{(x-y)^{2}}+\frac{\left [a, b\right](y)}{x-y},
\label{eq:KMOP}
\end{align}
This  example is known as an affine Kac-Moody algebra theory or as a current algebra in conformal field theory. 
The Heisenberg algebra  \eqref{eq:Heisalg} corresponds to a 1-dimensional subalgebra generated by $a$ for which $\langle a , a \rangle\neq 0$.

\subsubsection{The Virasoro Algebra}
Consider the vector space
\[
\Vir=\bigoplus_{n\in \Z}\C L_n\oplus \C K,
\]
 with basis $L_n$ and central element $K$ obeying the Virasoro algebra with commutation relations:
\begin{align}
\left[L_{m},L_n\right]=(m-n)L_{m+n}+\frac{1}{2}K\binom{m+1}{3}\delta_{m+n,0},\quad \left[L_{m},K\right]=0.
\label{eq:Viralg}
\end{align}
(The factor of $\frac{1}{2}$ is conventional.)
Define the formal series \footnote{$\omega(z)$ is usually notated by $T(z)$ in conformal field theory and is called the energy momentum tensor.}
\begin{align}
\omega(z)=\sum\limits_{n\in\Z}L_{n}z^{-n-2}.
\label{eq:TVir}
\end{align} 
Note that $\omega(z)$ is the formal series for a sequence with components
\begin{equation*}
\omega_{n}=L_{n-1}.
\end{equation*}
Recalling  the formal derivative  \eqref{delz} it follows that 
\begin{align*}
y^{-m}\left[\omega_{m},\omega(y)\right]&=\sum_{n\in\Z}\left[(m-n)\omega_{m+n-1}+\frac{1}{2}K\binom{m}{3}\delta_{m+n,2}\right] y^{-m-n-1}
\\
&=\partial\omega(y)+2\omega(y)\, \binom{m}{1}y^{-1}+\frac{1}{2}K\,\binom{m}{3}y^{-3},
\end{align*}
a polynomial in $m$ of degree $3$. 
Hence $\omega(z) $ is local of order $4$ from Theorem~\ref{theor:ablocal}~(v) and  with a suitable universal enveloping algebra we obtain the OPE 
\begin{align}
\omega(x)\omega(y)\OPE \frac{\frac{1}{2}K}{(x-y)^{4}}
+\frac{2\omega(y)}{(x-y)^{2}}
+\frac{\partial\omega(y)}{x-y}.
\label{eq:VirOP}
\end{align}
\bigskip

\section{Creative Fields} 
\label{chapt:FCT}
\subsection{Fields} 
Let $V$ be a vector space over $\C$. 
We shall often refer to an element of $V$ as a \emph{state}.  
Let $\End(V)$ denote the algebra  of 
endomorphisms of $V$ i.e. linear maps from $V$ to $V$.
$\End(V)$ is an associative algebra with unit given by the identity map $I_V$ over the field $\C$ with bilinear product given by the composition of linear maps $AB=A\circ B$ for all $A,B\in \End(V)$. 
\medskip

Consider a formal series $\alpha(z)=\sum_{n\in\Z}\alpha_nz^{-n-1}$ with components $\alpha_n\in\End(V)$.
$\alpha(z)$ is called a  \emph{field}  
if for any $v\in V$ 
\begin{align}
\alpha_n v=0\ \mbox{for} \ n\gg 0,
\label{eq:lowtrunc}
\end{align}
 i.e. for $n$ sufficiently large.
The property \eqref{eq:lowtrunc} is called \emph{lower truncation}  and plays a vital role in vertex algebras. 

\medskip

For  fields  $\alpha(x),\beta(y)$ we may extend the definition of the bivariate formal series \eqref{eq:Cxy0} to all $n\in \Z$ by defining
\begin{align}
C^{n}(\alpha(x),\beta(y))=(x-y)^{n}\alpha(x)\beta(y)-(-y+x)^{n}\beta(y)\alpha(x).
\label{eq:Cxy}
\end{align}
Clearly, this agrees with \eqref{eq:Cxy0} for $n\ge 0$. In general, applying the formal expansion convention  
of \eqref{eq:xypower} we find 
\begin{align*}
C^{n}(\alpha(x),\beta(y))=\sum_{l,m\in\Z}
C^{n}_{lm}(\alpha,\beta)x^{-l-1}y^{-m-1},
\end{align*}
where for all $l,m,n\in\Z$
\begin{align}
C^{n}_{lm}(\alpha,\beta)=\sum_{i\ge 0}(-1)^i\binom{n}{i}
\left(
\alpha_{l+n-i}\beta_{m+i}-(-1)^n\beta_{m+n-i}\alpha_{l+i}
\right).
\label{eq:ctrs}
\end{align}

\begin{remark}
\label{rem:crst}
 By lower truncation, $C^{n}_{lm}(\alpha,\beta)v$  reduces to a finite sum of terms for each $v\in V$ and hence  $C^{n}(\alpha(x),\beta(y))$ is a well-defined formal bivariate series.
\end{remark}
\begin{lemma}
\label{lem:Cnxyk}
For all $k\ge 0$ and all $n\in \Z$ we have 
\begin{align}
(x-y)^kC^n(\alpha(x),\beta(y))=C^{n+k}(\alpha(x),\beta(y)).
\label{eq:Cnxyk}
\end{align}
\end{lemma}
\begin{proof}
The expansion convention \eqref{eq:xypower} implies $(x-y)^k(x-y)^n=(x-y)^{n+k}$ and 
$(x-y)^k(-y+x)^n=(-y+x)^k(-y+x)^n=(-y+x)^{n+k}$ for $k\ge 0$.
\end{proof}
By Remark~\ref{rem:crst}, for fields $\alpha(x),\beta(y)$
we may similarly extend the definition of the $n^{\nth}$ residue product
${*}_{n}$   to  all $n\in \Z$ with
\begin{align}
\nprod{\alpha}{n}{\beta}(z)=
\Res_x C^{n}(\alpha(x),\beta(z))=
\sum_{i\ge 0}\binom{n}{i}
\left((-z)^i
\alpha_{n-i}\beta(z)-(-z)^{n-i}\beta(z)\alpha_{i}\right),
\label{eq:nthprod}
\end{align}
with components
\begin{align}
\nprod{\alpha}{n}{\beta}_{m}=C^{n}_{0m}(\alpha,\beta)=
\sum_{i\ge 0}(-1)^i\binom{n}{i}
\left(
\alpha_{n-i}\beta_{m+i}-(-1)^n\beta_{m+n-i}\alpha_{i}
\right).
\label{eq:nthassoc}
\end{align}
\begin{lemma}
\label{lem:nprodfield}
If $\alpha(z),\beta(z)$ are fields then
$\nprod{\alpha}{n}{\beta}(z)$ is a field for all $n\in\Z$. 
\end{lemma}
\begin{proof}
\eqref{eq:nthassoc} implies 
$\nprod{\alpha}{n}{\beta}(z)$ is a field provided
$C^{n}_{0m}(\alpha,\beta)v=0$ 
for any $v\in V$ for $m\gg 0$.  But $\beta_{m+i}v=0$ and  $\beta_{m+n-i}\alpha_{i}v=0$ for $m\gg 0$ for some $v$ dependent finite range of $i$ following  Remark~\ref{rem:crst}.
\end{proof} 

\noindent For $n<0$, $\nprod{\alpha}{n}{\beta}(z)$  is related to the normally ordered product  \eqref{eq:normab} as follows:
\begin{lemma}\label{lem:Resprodneg}
For fields $\alpha(x),\beta(z)$ and $k\ge 0$ we have
\begin{align}
\nprod{\alpha}{-k-1}{\beta}(z)=\nord{\partial^{(k)} \alpha(z)\beta(z)}.
\label{eq:abkmin1}
\end{align}
\end{lemma} 
\begin{proof}The result follows directly from Lemma~\ref{lem:Respm}. 
\end{proof} 

\begin{lemma}
\label{lem:dernprod}
$\partial$ is a derivation of the $n^{\nth}$  residue product of two fields $\alpha(z),\beta(z)$ i.e.  
\begin{align}
\partial \nprod{\alpha}{n}{\beta}(z)= \nprod{\partial\alpha}{n}{\beta}(z)+\nprod{\alpha}{n}{\partial \beta}(z).
\label{eq:dernprod}
\end{align}
\end{lemma}
\begin{proof} 
From \eqref{eq:Cxy} we directly find 
\begin{align*}
\left(\partial_x + \partial_z\right)C^{n}(\alpha(x),\beta(z))=C^{n}(\partial\alpha(x),\beta(z))+C^{n}(\alpha(x),\partial\beta(z)).
\end{align*}
Taking $\Res_x$, the result follows since 
$\Res_x\partial_x C^{n}(\alpha(x),\beta(z))=0$  from \eqref{eq:FTC}.
\end{proof}

The next theorem is fundamental to the theory of vertex algebras.
It is often stated either as a foundational axiom \cite{B, FLM, Li} or else is proved subject to some further assumed properties \cite{Li, K, MN}. However, here we only assume that $\alpha(x),\beta(y)$ are local fields.
\begin{theorem}[Borcherds-Frenkel-Lepowsky-Meurmann identity]
\label{th:BFLM}
Let $\alpha(x),\beta(y)$ be    mutually local fields. Then  for all $l,m,n\in\Z$ we have 
\begin{align}
\sum_{i\ge 0}\binom{l}{i}\nprod{\alpha}{n+i}{\beta}_{l+m-i}=
\sum_{i\ge 0}(-1)^i\binom{n}{i}
\left(
\alpha_{l+n-i}\beta_{m+i}-(-1)^n\beta_{m+n-i}\alpha_{l+i}
\right).
\label{eq:BFLM}
\end{align}
\end{theorem}
\begin{proof} 
Let $\alpha(z)\loc{N}\beta(z)$ for $N\ge 0$. Thus
\eqref{eq:BFLM} is the trivial identity $0=0$ for $n\ge N$ by locality and \eqref{eq:nthprodzero} so that we need only consider $n<N$. 
In a similar fashion to the proof of
Theorem~\ref{theor:ablocal}, the identity \eqref{eq:BFLM} is a consequence of Newton  forward differences applied to an appropriate  choice of sequence. 
Note that the right hand side of \eqref{eq:BFLM} is $C^{n}_{lm}(\alpha,\beta)$ of \eqref{eq:ctrs}.
For each $n<N$ we define a sequence $\gamma^{n}$ with components in $W:=\End(V)[[y,y^{-1}]]$ labelled by $l\in\Z$ as follows
\begin{align*}
\left(\gamma^{n}\right)_l=y^{-l}\sum_{m\in\Z}C^{n}_{lm}(\alpha,\beta)y^{-m-1},
\end{align*}
with formal series
\begin{align*}
\gamma^{n}(z)=y\, C^n(\alpha(yz),\beta(y)).
\end{align*} 
Since $\alpha(x)\loc{N}\beta(y)$ and using \eqref{eq:Cnxyk}   we find 
for each $n<N$ that
\begin{align*}
(z-1)^{N-n}\gamma^{n}(z)=y^{1+n-N}C^{N}(\alpha(yz),\beta(y))=0.
\end{align*}
Thus applying Newton's forward difference formula Theorem~\ref{theor:alphaseq}~(iii) 
we find
\begin{align}
\left(\gamma^{n}\right)_l=\sum_{i\ge 0}\binom{l}{i}R_i^{n},
\label{eq:gyNewt}
\end{align}
with  $R_i^{n}$ for $i\ge 0$ given by
\begin{align*}
R_i^{n}&=y \Res_z(z-1)^i C^{n}(\alpha(yz),\beta(y))
\\
&=y^{1-i} \Res_z  C^{n+i}(\alpha(yz),\beta(y))\\
&=y^{-i} \Res_x  C^{n+i}(\alpha(x),\beta(y))
=y^{-i}\nprod{\alpha}{n+i}{\beta}(y), 
\end{align*}
using \eqref{eq:Cnxyk} and that 
$\Res_z \rho(yz)=y^{-1}\Res_x \rho(x)$ 
for any formal series $\rho(x)$. 
We have therefore shown that $\left(\gamma^{n}\right)_l y^l $ is given by
\begin{align}
 \sum_{m\in\Z}C^{n}_{lm}(\alpha,\beta)y^{-m-1}
=\sum_{i\ge 0}\binom{l}{i}y^{l-i}\nprod{\alpha}{n+i}{\beta}(y).
\label{eq:BFLM2}
\end{align}
The result follows on computing the  coefficients of 
$y ^{-m-1}$ in \eqref{eq:BFLM2}. 
\end{proof} 
\medskip

\label{rem:BFLM}
\eqref{eq:BFLM}  specializes to the commutator formula Theorem~\ref{theor:ablocal}~(iv) for $n=0$ and to the  residue  product formula \eqref{eq:nthprod} when $l=0$.
There are a number of equivalent ways of writing \eqref{eq:BFLM}. 
\begin{proposition}
\label{prop:BFLM2}
Both of the following identities are equivalent to the Borcherds-Frenkel-Lepowsky-Meurmann identity:
\begin{align}
\sum_{i\ge 0}y^{-i-1}\delta^{(i)}\left(\frac{x}{y}\right)\nprod{\alpha}{n+i}{\beta}(y)&=
(x-y)^{n}\alpha(x)\beta(y)-(-y+x)^{n}\beta(y)\alpha(x),
\label{eq:BFLM3}
\\
y^{-1}\delta\left( \frac{x-z}{y}\right)\sum_{m\in \Z}\nprod{\alpha}{m}{\beta}(y)z^{-m-1}
&=z^{-1}\delta\left( \frac{x-y}{z}\right)\alpha(x)\beta(y)- z^{-1}\delta\left( \frac{-y+x}{z}\right)\beta(y)\alpha(x).
\label{eq:BFLM4}
\end{align}  
\end{proposition} 
\begin{proof}
\eqref{eq:BFLM2} is equivalent to 
\begin{align*}
C^{n}(\alpha(x),\beta(y))&=
\sum_{l\in\Z}x^{-l-1}\sum_{i\ge 0}\binom{l}{i}y^{l-i}\nprod{\alpha}{n+i}{\beta}(y).
\end{align*}
Recalling \eqref{eq:deltai} this can be written as \eqref{eq:BFLM3}. This  in turn is equivalent to 
\begin{align*}
&\sum_{n\in \Z}z^{-n-1}\sum_{i\ge 0}y^{-i-1}\delta^{(i)}\left(\frac{x}{y}\right)\nprod{\alpha}{n+i}{\beta}(y)
\\
&=
z^{-1}\sum_{n\in \Z}\left(\frac{x-y}{z}\right)^{n}\alpha(x)\beta(y)
-z^{-1}\sum_{n\in \Z}\left(\frac{-y+x}{z}\right)^{n}\beta(y)\alpha(x)
\\
&=z^{-1}\delta\left( \frac{x-y}{z}\right)\alpha(x)\beta(y)- z^{-1}\delta\left( \frac{-y+x}{z}\right)\beta(y)\alpha(x),
\end{align*}
recalling \eqref{eq:delta}. Finally, Taylor's Theorem of Lemma~\ref{lem:Taylor} and \eqref{deltazi}  imply
\begin{align*}
\delta\left( \frac{x-z}{y}\right)=\sum_{i\ge 0}\delta^{(i)}\left(\frac{x}{y}\right)\left( \frac{z}{y}\right)^{i},
\end{align*}
so that, after relabelling, we obtain 
\begin{align*}
\sum_{n\in \Z}z^{-n-1}\sum_{i\ge 0}y^{-i-1}\delta^{(i)}\left(\frac{x}{y}\right)\nprod{\alpha}{n+i}{\beta}(y)=y^{-1}\delta\left( \frac{x-z}{y}\right)\sum_{m\in \Z}\nprod{\alpha}{m}{\beta}(y)z^{-m-1}.
\end{align*}
Thus the result holds.
\end{proof}
\begin{remark}
\label{rem:BFLMlocal}
For $n\ge 0$ the Borcherds-Frenkel-Lepowsky-Meurmann identity \eqref{eq:BFLM3} follows from locality using Theorem~\ref{theor:ablocal}~(vi) and Lemma~\ref{lem:deltafun}. 
\end{remark}
\medskip

\noindent The next result is very useful for the  construction of local fields. 
\begin{lemma}[Dong's Lemma]
\label{lem:Dong}
Let $\alpha(z),\beta(z),\gamma(z)$ be mutually local fields. 
Then $\nprod{\alpha}{n}{\beta}(z)$ and $\gamma(z)$ are mutually local fields for all $n\in \Z$.
\end{lemma}
\begin{proof}
For some orders of locality $K,L,M\ge 0$ we have
\begin{align*}
\alpha(z)\loc{K}\beta(z),\quad \alpha(z)\loc{L}\gamma(z),\quad \beta(z)\loc{M}\gamma(z).
\end{align*} 
In particular, $C^{n}(\alpha(x),\beta(z))=0$ and $\nprod{\alpha}{n}{\beta}(z)=0$ for $n\ge K$. Hence we need only consider $n\le K-1$. Let $N= K+L+M-n-1$ and define 
\begin{align*}
D (x,y,z)&=(y-z)^N\left[\gamma(y),C^{n}(\alpha(x),\beta(z))\right].
\end{align*}
Note that $N\ge 0 $ since $L,M,K-n-1\ge 0$.
Using \eqref{eq:Cnxyk}  we find
\begin{align*}
D (x,y,z)&=(y-z)^M(y-x+x-z)^{N-M} \left[\gamma(y), C^{n}(\alpha(x),\beta(z))\right]
\\
&= 
(y-z)^M \sum_{r=0}^{K-n-1}\binom{N-M}{r}(y-x)^{N-M-r}\,\left[\gamma(y),C^{n+r}(\alpha(x),\beta(z))\right],
\end{align*}
where $r\le K-n-1$ in the sum  since $C^{n+r}(\alpha(x),\beta(z))=0$ for $ n+r\ge K $. 
Therefore  $ N-M-r\ge  L $   for each $r$ in the sum so that
\[
(y-z)^M  (y-x)^{N-M-r}\,\left[\gamma(y),C^{n+r}(\alpha(x),\beta(z))\right]=0,
\] 
since $\alpha(z)\loc{L}\gamma(z)$ and $ \beta(z)\loc{M}\gamma(z)$.
Thus  $D(x,y,z)=0$ which implies 
\[
C^N(\gamma(y),\nprod{\alpha}{n}{\beta}(z))=\Res_{x}D (x,y,z)=0,
\]
i.e. $\gamma(z)\loc{}\nprod{\alpha}{n}{\beta}(z)$ with order of locality at most $ N$.
\end{proof}

\subsection{Creative Fields} 
Let $\vac\in V$ denote  a distinguished state  called the \emph{vacuum vector}.\footnote{The vacuum vector is usually denoted by $\vert 0\rangle$ in CFT.} A \emph{creative  field} for  $a\in V$ is a field
which we notate by
\[
a(z)=\sum_{n\in \Z}a_n z^{-n-1},
\]
with components or \emph{modes} $a_n\in\End(V)$ such that
\begin{align}
a_{-1}\vac &= a, \label{eq:creativity1}\\
\quad a_{n}\vac &= 0,\mbox{ for all }n\ge 0. \label{eq:creativity2}
\end{align}
\eqref{eq:creativity2} is equivalent to $a(z)_{+}\vac =0$ (cf. \eqref{eq:alphapm}).
We sometimes write \eqref{eq:creativity1} and \eqref{eq:creativity2} together as 
\footnote{ 
This is usually written in CFT as $\lim_{z\rightarrow 0} a(z)\vert 0\rangle=a$.}
\[
a(z)\vac=a+O(z)\in V[[z]],
\]
where $V[[z]]$ is the space of formal power series in $z$ with coefficients in $V$.
\medskip

The following lemma describes several important examples of creative fields.
\begin{lemma}\label{lem:creative_examples}
Let $a(z), b(z)$ be creative fields for states $a,b\in V$, respectively.
\begin{enumerate}[(i)]
	\item $za(z)$ creates the zero vector $0$,
	\item  $I(z)=\Id_V$, the identity $V$ endomorphism, creates the vacuum $\vac$,
	\item  $\lambda a(z)+\mu b(z)$ creates $\lambda a+\mu b$ for $\lambda,\mu\in \C$,
	\item  $\nprod{a}{n}{b}(z)$ creates $a_{n}b$ for $n\in\Z$,
	\item $\nord{\partial^{(k)}a(z) b(z)}$ creates $a_{-k-1}b$ for $k\ge 0$,
	\item  $ \partial^{(k)} a(z)$ creates $a_{-k-1}\vac$.
\end{enumerate}
\end{lemma}
\begin{proof}(i)--(iii) are trivially true. 
\eqref{eq:nthprod}  implies that 
\begin{align*}
\nprod{a}{n}{b}(z)\vac &=\sum_{i\ge 0}\binom{n}{i}
\left((-z)^i\,
a_{n-i}b(z)-(-z)^{n-i}b(z)a_{i}\right)\vac\\
&=\sum_{i\ge 0}\binom{n}{i} (-z)^i\, a_{n-i}\left(b +O(z)\right)= a_{n}b+O(z),
\end{align*}
using creativity of  $a(z)$ and $b(z)$. Hence (iv) holds.  (iv) implies (v) on using  Lemma~\eqref{lem:Resprodneg}. (vi) follows from (ii) and (v) on choosing $b=\vac$ and $b(z)=I(z)$.
\end{proof} 
\begin{remark}\label{rem:nonunique}
A creative field $a(z)$ for $a\in V$ is clearly  not unique   since, by Lemma~\ref{lem:creative_examples}~(i), $ a(z)+zb(z) $ also creates $a$ for any creative field $b(z)$.
\end{remark}
\noindent The lower truncation property \eqref{eq:lowtrunc} is refined for local creative fields as follows:
\begin{corollary}[Lower Truncation]\label{cor:lowtrunc}
Let $a(z), b(z)$ be local creative fields for $a,b\in V$ respectively. Then 
$a(z) \loc{N} b(z)$ implies
\begin{align} a_{n}b=0 \mbox{ for all } n\ge N.
\label{eq:lowertrun}
\end{align}
\end{corollary}
\begin{proof}
$\nprod{a}{n}{b}(z)=0$ for $n\ge N$ by  \eqref{eq:nthprodzero}  so that 
Lemma~\ref{lem:creative_examples}~(v) implies the result.   
\end{proof} 

\section{Vertex Algebras} 
\subsection{Uniqueness and Translation Covariance}
Consider a vector space $V$ with vacuum vector $\vac\in V$ and a set of mutually local creative fields $\mathcal{F}:=\{a(z):a\in V\}$.
By Remark~\ref{rem:nonunique}, $a(z)\in \mathcal{F}$ is not the unique creative field for $a\in V$. 
\begin{proposition}\label{goddard}
Suppose that $\phi(z)\in \mathcal{F}$ is a creative field for the zero state $0$. Then 
\begin{equation}
\phi(z)\vac=0 \Leftrightarrow \phi(z)=0.
\label{eq:phi0}
\end{equation}
\end{proposition}
\begin{proof}
Assume that  $\phi(z)\vac=0$. Let  $a\in V$ with a creative field $a(z)\in \mathcal{F}$  where $a(z)\loc{N}\phi(z)$ for some $N\ge 0$. Then 
\begin{align*}
0 &=x^{-N}C^N (\phi(x),a(y))\vac =x^{-N}(x-y)^N \phi(x) a(y)\vac= \phi(x) a+O(y),
\end{align*}
i.e. $\phi(x) a=0$. This is true for any  $a\in V$ so that $\phi(x)=0$. The converse is trivial.
\end{proof}
\noindent This result immediately implies:
\begin{corollary} 
\label{cor:aunique} Let $a(z),\widetilde{a}(z)\in \mathcal{F}$ be  creative fields 
for $a\in V$. Then 
\[
a(z)=\widetilde{a}(z)\Leftrightarrow  a(z)\vac=\widetilde{a}(z)\vac.
\]
\end{corollary}
We now describe  a uniqueness criterion for $\mathcal{F}$. Let $T\in \End(V)$ such that
\begin{align}
T\vac &=0 \label{eq:Tvac},
\\
\left[ T,a(z)\right]&=\partial a(z)  \mbox{ for all } a(z)\in \mathcal{F}.
\label{eq:Translation}
\end{align}
In terms of modes, \eqref{eq:Translation} is equivalent to
\begin{align}
[T,a_n]=-na_{n-1}.
\label{eq:Tan}
\end{align}
$T$ is called a \emph{translation operator} and  $\mathcal{F}$ is said to be \emph{translation covariant} if \eqref{eq:Tvac} and \eqref{eq:Translation} are satisfied for a translation operator $T$.
\begin{theorem}[Uniqueness]
\label{theor:uniqueT}
Let $\mathcal{F}$ be a set of mutually local creative fields for $V$. The elements of $\mathcal{F}$ are unique if and only if $\mathcal{F}$ is translation covariant.
\end{theorem}  
\begin{proof}
Assume that the elements of $\mathcal{F}$ are unique. Define $T\in \End(V)$ by
\begin{align}
Ta=a_{-2}\vac,
\label{eq:Tadef}
\end{align}
 for each $a\in V$ with unique creative field $a(z)$. By Lemma~\ref{lem:creative_examples}~(ii) we know that $I(z)=\Id_{V}$ is a creative field for $\vac$ and is therefore unique by assumption. Thus \eqref{eq:Tadef} implies \eqref{eq:Tvac}. 
By Dong's Lemma~\ref{lem:Dong} and Lemma~\ref{lem:creative_examples}~(iv) we also know that  $\nprod{a}{n}{b}(z)\in \mathcal{F}$ is a creative   field for $a_{n}b$ for each $a,b\in V$. Hence, by the assumed uniqueness property 
\begin{align}
\left(a_{n}b\right)(z)=\nprod{a}{n}{b}(z).
\label{eq:nprodvertex}
\end{align}
In particular, using \eqref{eq:nthassoc}  we find that for all $a,b\in V$
\begin{align*}
T(a_{n}b)=(a_n b)_{-2}\vac &= \sum_{i\ge 0}(-1)^i\binom{n}{i}
\left(
a_{n-i}b_{i-2}-(-1)^n b_{ n-i-2}\,a_{i}
\right) \vac
\\
&= a_{n}b_{-2}\vac -na_{n-1} b_{-1}\vac=a_{n}Tb-na_{n-1}b.
\end{align*}
 Hence $\mathcal{F}$ is translation covariant using \eqref{eq:Tan}.

Conversely, assume that $\mathcal{F}$ is translation covariant with some translation operator $T$. 
Thus for  $a(z)\in \mathcal{F}$, 
\eqref{eq:Tvac} and \eqref{eq:Translation} imply that $Ta_{-k}\vac=ka_{-k-1}\vac$ for all $k\in \Z$.
Hence $T^n  a=T^n a_{-1}\vac=n!  a_{-n-1}\vac$ for all $n\ge 0$ so that
\begin{align}
a(z)\vac= e^{zT}a.
\label{eq:expzTa}
\end{align} 
But if $\widetilde{a}(z)$ is another translation covariant  creative field for $a$ then $\widetilde{a}(z)\vac= e^{zT}a=a(z)\vac$. 
Hence by Corollary~\ref{cor:aunique} we conclude that $a(z)=\widetilde{a}(z)$. Therefore the elements of $\mathcal{F}$ are unique.
\end{proof}

\subsection{Vertex  Algebras}
We have now gathered all the requisite concepts to define a vertex algebra. 
Let $Y(a,z)$ denote the unique translation covariant creative field for $a\in V$ of Theorem~\ref{theor:uniqueT}. $Y(a,z)$ is called the \emph{vertex operator} for $a$.  
$Y$ can also be construed as a mapping 
\begin{align}
Y:V&\rightarrow \End(V)[[z,z^{-1}]],
\notag\\ 
a&\mapsto Y(a,z)=\sum_{n\in\Z}a_{n}z^{-n-1},
\label{eq:Ylinear}
\end{align} 
called the \emph{state-field correspondence}.
\begin{definition}\label{defn:VA}\index{vertex algebra}
A \emph{Vertex Algebra} consists of the data $(V,Y,T,\vac)$ where $V$ is a vector space,  a distinguished vacuum vector $\vac\in V$, a translation operator $T\in \End(V)$ and a state-field correspondence $Y$  with the following properties:
\begin{align*}
&\mbox{\textbf{locality:}}\quad  Y(a, z) \sim Y(b, z)\mbox{ for all } a,b\in V, \\
&\mbox{\textbf{creativity:}}\quad  Y(a, z)\vac = a + O(z),\\
&\mbox{\textbf{translation covariance:}} \quad 
\left[ T,Y(a,z)\right]=\partial Y(a,z),\quad T\vac =0 .
\end{align*}
\end{definition}
\begin{lemma}
The state-field correspondence is an injective linear map.
\end{lemma}
\begin{proof}
Linearity follows from Lemma~\ref{lem:creative_examples}~(iii).
Suppose  $Y(a,z)=Y(b,z)$ for $a,b\in V$. Then $a_{-1}=b_{-1}$ so that  $a=a_{-1}\vac=b_{-1}\vac=b$. Hence $Y$ is injective.
\end{proof}

We describe a number of important properties of vertex operators: 
\begin{proposition}
\label{prop:vopsexamples}
Let $a(z)=Y(a,z)$ and $ b(z)=Y(b,z)$ be the vertex operators for   $a,b\in V$. 
\begin{enumerate}[(i)]
	\item $Y(\vac,z)=\Id_V$,
	\item $Y(a,z)\vac =e^{z T}a$,
	\item  $Y(a_{n}b,z)=\nprod{a}{n}{b}(z)$  for all $n\in\Z$,
	\item $Y(Ta,z)=\partial Y(a,z)$.
\end{enumerate} 
\end{proposition}
\begin{proof}
$\Id_V\in \mathcal{F}$ creates $\vac$, by Lemma~\ref{lem:creative_examples}~(ii), and is  translation covariant giving (i).
Property~(ii) was shown in \eqref{eq:expzTa} in the proof of the Uniqueness Theorem~\ref{theor:uniqueT}. 
 $\nprod{a}{n}{b}(z)\in \mathcal{F}$ is a local creative field for $a_{n}b$ by 
Lemma~\ref{lem:nprodfield} and Lemma~\ref{lem:creative_examples}~(iii). 
Translation covariance of $a(z)$ and $b(z)$ implies  
\[
\left[ T,C^n(a(x),b(z))\right]=C^n(\partial a(x),b(z))+C^n(a(x),\partial b(z)),
\]
so that 
\begin{align}
\label{eq:Tanb}
\left[ T,\nprod{a}{n}{b}(z)\right]&= \Res_{x} \left[ T,C^n(a(x),b(z))\right]
\notag
\\
&= \nprod{\partial a}{n}{b}(z) + \nprod{a}{n}{\partial b}(z)
=\partial \nprod{a}{n}{b}(z),
\end{align}
by Lemma~\ref{lem:dernprod}. Thus $\nprod{a}{n}{b}(z)$ is translation covariant  and so (iii) holds.

Lemma~\ref{lem:delalpha} and Lemma~\ref{lem:creative_examples}~(vi) imply $\partial Y(a,z)$ is a local creative field for $a_{-2}\vac=Ta$. Translation covariance for $Y(a,z)$ implies
\begin{align*}
\left [T,\partial Y(a,z) \right]&=\partial \left [T, Y(a,z) \right]=\partial (\partial Y(a,z)),
\end{align*}
so that $\partial Y(a,z)$ is also translation covariant. 
Therefore (iv) follows from the Uniqueness Theorem~\ref{theor:uniqueT}. 
\end{proof}
\begin{corollary}\label{cor:Tder}
$T$ is a derivation of the vertex algebra where for all $a,b\in V$: 
\[
T(a_{n}b)=(Ta)_{n} b +a_{n} Tb.
\]
\end{corollary}
\begin{proof}
$(Ta)_{n}=-na_{n-1}$ from  Proposition~\ref{prop:vopsexamples}~(iii). 
\eqref{eq:Tan} implies $[T,a_n]b=(Ta)_{n} b$.
\end{proof}

Proposition~\ref{prop:vopsexamples}~(iii) implies that, for a vertex algebra, we may replace all $n^{\nth}$ residue products $\nprod{a}{n}{b}(z)$  by the unique vertex operator $Y(a_{n}b,z)$ in the previous sections. Thus  the locality Theorem~\ref{theor:ablocal} implies the \emph{Commutator Formulas}
\begin{align}
\label{eq:VAcom1}
[a_m,Y(b,z)] &=\sum\limits_{i\ge 0}  \binom{m}{i}Y(a_{i}b,z) z^{m-i},
\\
\label{eq:VAcom2}
[a_m,b_n] &=\sum\limits_{i\ge 0}  \binom{m}{i}(a_{i}b)_{m+n-i},
\end{align}
for all $m,n\in \Z$. Similarly \eqref{eq:nthassoc} implies the A\emph{ssociator Formula}
\begin{align}
(a_nb)_{m}=\sum_{i\ge 0}(-1)^i\binom{n}{i}
\left(
a_{n-i}b_{m+i}-(-1)^n b_{m+n-i}a_{i}
\right).
\label{eq:VAassoc}
\end{align}
These can be combined into the Borcherds-Frenkel-Lepowsky-Meurmann identity 
\begin{align}
\sum_{i\ge 0}\binom{l}{i}\left(a_{n+i}b\right)_{l+m-i}&=
\sum_{i\ge 0}(-1)^i\binom{n}{i}
\left(
a_{l+n-i}b_{m+i}-(-1)^n b_{m+n-i}a_{l+i}
\right),
\label{eq:VABFLM1}
\end{align}
for all $ l,m,n\in\Z$ (cf. \eqref{eq:BFLM}). This in turn is equivalent to (cf.  \eqref{eq:BFLM4}) 
\begin{align}
&z^{-1}\delta\left( \frac{x-y}{z}\right)Y(a,x)Y(b,y)
- z^{-1}\delta\left( \frac{-y+x}{z}\right)Y(b,y)Y(a,x)
\notag
\\
&=y^{-1}\delta\left( \frac{x-z}{y}\right)Y(Y(a,z)b,y).
\label{eq:VABFLM2}
\end{align}  
\eqref{eq:VAcom2} and \eqref{eq:VAassoc} are axioms in the original formulation of vertex algebras by Borcherds in \cite{B}. These were shown to be equivalent to the identity \eqref{eq:VABFLM2}, called the Jacobi identity by Frenkel, Lepowsky and Meurmann \cite{FLM}.

\subsection{Translation and Skewsymmetry}
\begin{lemma}[Translation Symmetry]\label{lem:translation}
$T$ is a generator of translation symmetry:
\begin{align*}
e^{yT} Y(a,x) e^{-y T}=Y(a,x+y).
\end{align*}
\end{lemma}
\begin{proof} 
The Baker-Campbell-Hausdorff formula for linear operators $A,B$ states that  
\[
e^{A}Be^{-A}=e^{\ad_{A}}B,
\]
where $\ad_{A}(\cdot)=[A,\cdot]$  is the adjoint operator. Thus we find
\begin{align*}
e^{yT} Y(a,x) e^{-y T}=e^{y\ad_{T}}Y(a,x)=e^{y\partial}Y(a,x),
\end{align*} 
by translation covariance \eqref{eq:Translation}. The result follows from Taylor's theorem \eqref{eq:Taylor}.
\end{proof}

\begin{lemma}[Skew-Symmetry]
\label{lem:skew}
Let $a,b\in V$, a vertex algebra. Then 
\begin{align}
Y(a,z)b=e^{zT}Y(b,-z)a,
\label{eq:skew}
\end{align}
or in terms of components:
\begin{align}
a_{n}b{=} (-1)^{n+1}\sum_{k\geq 0}(-1)^k T^k b_{n{+}k}a.
\label{eq:skew2}
\end{align} 
\end{lemma}
\begin{proof}
Let $Y(a,z)\loc{N} Y(b,z)$ so that
\[
(z-y)^{N}Y(a,z)Y(b,y)\vac =(z-y)^{N}Y(b,y)Y(a,z)\vac.
\]
By Proposition~\ref{prop:vopsexamples}~(ii) and translation symmetry we have
\begin{align}
(z-y)^{N}Y(a,z)e^{yT}b &= (z-y)^{N}Y(b,y) e^{z T} a
\notag
\\
&=(z-y)^{N} e^{z T}Y(b,y-z) a.
\label{eq:skewzy}
\end{align}
Lower truncation \eqref{eq:lowertrun} implies  $x^{N} Y(b,x) a$ contains no  negative powers of $x$. Thus $(z-y)^{N}Y(b,y-z) a$ also contains no negative powers of $y$. Taking $y=0$ in \eqref{eq:skewzy} we obtain  \eqref{eq:skew} on multiplying by $z^{-N}$.
\eqref{eq:skew2} follows immediately.
\end{proof}

\subsection{Examples of Vertex Algebras}
We have the following very useful generating theorem \cite{FKRW}, \cite{MP}.
 \begin{theorem}[Generating Theorem]
 \label{th:gen}
 Let $V$ be a vector space with $ \vac  \in V$
and $T\in\End(V)$. Let $\{a^{i}(z)\}_{i\in \mathcal{I}}$ for some indexing set $\mathcal {I}$ be a set of mutually local, creative, translation-covariant fields which \emph{generates} $V$ i.e.
\begin{align*}
V = \mbox{span}\{a^{i_1}_{n_1} \hdots a^{i_k}_{n_k} \vac  \ |   n_1, \hdots , n_k \in \Z, i_{1},\ldots i_{k}\in \mathcal{I} \}.
\end{align*}
Then there is a unique vertex algebra $(V, Y, \vac , T)$ with vertex operators defined on the spanning set by
\begin{align}
Y\left(a^{i_1}_{n_1} \hdots a^{i_k}_{n_k} \vac,z\right)=
a^{i_1}*_{n_1}\left(a^{i_1}*_{n_2} \left(\hdots \left(a^{i_k}*_{n_k} I\right)\right)\right)(z),
\label{eq:gen}
\end{align}
a composition of $k$ residue products and where $I(z)=Y(\vac,z)=\Id_{V}$.
\end{theorem}
\begin{proof}
$\mathcal{F}=\{
a^{i_1}*_{n_1}\left(a^{i_1}*_{n_2} \left(\hdots \left(a^{i_k}*_{n_k} I\right)\right)\right)(z)\}$ 
is a set of mutually local creative fields for $V$ by repeated use  of Lemma~\ref{lem:nprodfield}, Dong's Lemma~\ref{lem:Dong} and Lemma~\ref{lem:creative_examples}~(iv).  
Furthermore, $a^{i_1}*_{n_1}\left(a^{i_1}*_{n_2} \left(\hdots \left(a^{i_k}*_{n_k} I\right)\right)\right)(z)$ is translation covariant by  \eqref{eq:Tanb}. 
Hence, by the Uniqueness Theorem~\ref{theor:uniqueT}, $\mathcal{F}$ forms a set of unique vertex operators on the spanning set and therefore by linearity on $V$.
\end{proof}
\subsubsection{The Heisenberg Vertex Algebra} The Heisenberg vertex algebra  is constructed from the Verma module\footnote{e.g. See \cite{K}, \cite{MT} for further details} $M_0$ of the Heisenberg Lie algebra \eqref{eq:Heisalg} given by 
\[
M_0=\mbox{span}\{h_{-n_1}\ldots h_{-n_k}v_{0}|  n_1, \hdots , n_k\ge 1\},
\]  
where $h_{n}v_{0}=0$ for all $n\ge 0$ and $Kv_{0}=v_{0}$. Then with $V=M_{0}$ and  $\vac=v_{0}$ we find $h(z)$ is a creative field for $h=h_{-1}\vac$ which is translation covariant for
\[
T=\sum_{n\ge 0}h_{-n-1}h_{n}.
\]
Thus Theorem~\ref{th:gen} and Lemma~\ref{lem:creative_examples} imply that $h(z)$ generates a vertex algebra with 
\[
Y(h_{-n_1}\ldots h_{-n_k}\vac,z)=
\nord{\partial^{(n_{1})}h(z) \nord{\partial^{(n_{2})}h(z)\ldots
\nord{\partial^{(n_{k-1})}h(z)\partial^{(n_k)}h(z)}}\ldots},
\]
for $n_1, \hdots , n_k\ge 1$.
\subsubsection{The Virasoro Vertex Algebra} The Virasoro vertex algebra  is constructed from a Verma module $M_{C,0}$ of the Virasoro Lie algebra \eqref{eq:Viralg} defined by 
\[
M_{C,0}=\mbox{span}\{L_{-n_1}\ldots L_{-n_k}v_{0}|  n_1, \hdots , n_k\ge 1\},
\]  
where $L_{n}v_{0}=0$ for all $n\ge 0$ and $Kv_{0}=Cv_{0}$. 
Then $\omega(z)=\sum_{n\in \Z}L_{n} z^{-n-2}$ is translation covariant for $T=L_{-1}$ but is \emph{not a creative field} with vacuum $v_0$ since  
\[
\omega(z) v_0=z^{-1}L_{-1}v_0+L_{-2}v_0 +O(z).
\] 
But since $L_{1}L_{-1}v_{0}=0$ it follows that 
\[
M_{C,1}=\mbox{span}\{L_{-n_1}\ldots L_{-n_k}L_{-1}v_{0}|  n_1, \hdots , n_k\ge 1\},
\]  
is submodule of $M_{C,0}$.  
Abusing notation by identifying states, operators and fields associated with 
 $M_{C, 0}$ with the corresponding states, operators and fields induced on the quotient
 $\Vir=M_{C, 0}/M_{C,1}$ we find that $Y(\omega, z) = \omega(z)$ generates a vertex algebra  with $T=L_{-1}$,   $\vac=v_0$ and $V=\Vir$ with vertex operators 
\[
Y(L_{-n_1}\ldots L_{-n_k}\vac,z)=
\nord{\partial^{(n_{1})}L(z) \nord{\partial^{(n_{2})}L(z)\ldots
\nord{\partial^{(n_{k-1})}L(z)\partial^{(n_k)}L(z)}}\ldots},
\]
for $n_1, \hdots , n_k\ge 2$.

\end{document}